\documentclass{amsart}
\usepackage[utf8x]{inputenc}
\usepackage{amsmath}
\usepackage{url}
\usepackage{tikz}
\usepackage{amsfonts}
\usetikzlibrary{calc,intersections}

\DeclareMathOperator{\conv}{conv}
\DeclareMathOperator{\xc}{xc}
\DeclareMathOperator{\nnegrk}{rk_{+}}
\newcommand{\setDef}[2]{\{{#1}\,:\,{#2}\}}
\newcommand{\R}{\mathbb{R}}
\newcommand{\Z}{\mathbb{Z}}
\newcommand{\N}{\mathbb{N}}
\newcommand{\scalProd}[2]{\langle{#1},{#2}\rangle}

\newtheorem{prop}{Proposition}
\newtheorem{lem}[prop]{Lemma}
\newtheorem{thm}[prop]{Theorem}

\newcommand{\n}{n} 
\newcommand{\di}{d} 
\newcommand{\m}{m} 
\newcommand{\otherdi}{d'} 

\title{Small Extended Formulations for Cyclic Polytopes}

\author{Yuri Bogomolov}
\address{P.G. Demidov Yaroslavl State University, ul. Sovetskaya, 14, Yaroslavl 150000, Russia} 
\email{mathematics@inbox.ru}

\author{Samuel Fiorini}
\thanks{S. Fiorini was partially funded by F.R.S.-FNRS (scientific mission OUT)}
\address{Universit\'{e} libre de Bruxelles, D\'{e}partement de Math\'{e}matique, Boulevard du Triomphe, B-1050 Brussels, Belgium}
\email{sfiorini@ulb.ac.be}

\author{Aleksandr Maksimenko}
\thanks{ A. Maksimenko was supported by the project No. 477 of P.G. Demidov Yaroslavl State University within State Assignment for Research}
\address{P.G. Demidov Yaroslavl State University, ul. Sovetskaya, 14, Yaroslavl 150000, Russia} 
\email{maximenko.a.n@gmail.com}

\author{Kanstantsin Pashkovich}
\thanks{K. Pashkovich was funded by F.R.S.-FNRS (research project T.0100.13, Semaphore 14620017)}
\address{University of Waterloo, Department of Combinatorics and Optimization, University Avenue West 200, Waterloo N2L 3G1, Ontario, Canada} 
\email{kanstantsin.pashkovich@gmail.com}

\usepackage{fancyhdr}
\begin{document}

\begin{abstract}
We provide an extended formulation of size~$O(\log \n)^{\lfloor\frac{\di}{2}\rfloor}$ for the cyclic polytope with dimension $\di$ and $\n$~vertices $(i,i^2,\ldots,i^{\di})$, $i \in [n]$. First, we find an extended formulation of size $\log(\n)$ for $\di = 2$. Then, we use this as base case to construct small-rank nonnegative factorizations of the slack matrices of higher-dimensional cyclic polytopes, by iterated tensor products. Through Yannakakis's factorization theorem, these factorizations yield small-size extended formulations for cyclic polytopes of dimension $\di \geq 3$.
\end{abstract}

\maketitle

\section{Introduction}  
Extended formulations is a rapidly developing field with connections to, among others, discrete mathematics and theoretical computer science. 
Two main reasons that make this field interesting are the facts that: (i) small-size extended formulations allow efficient formulations of various optimization problems; (ii) lower bounds on the sizes of extended formulations show fundamental limits to what can be efficiently expressed through linear programs, and more generally conic programs. Here our focus is on linear programming extended formulation, thus the underlying cone is the nonnegative orthant.

In this paper, we provide an extended formulation of size (at most)~$2(2\lfloor \log_2(n-1)\rfloor+2)^{\lfloor \di/2 \rfloor}$ for the $\di$-dimensional $\n$-vertex cyclic polytope with vertices $(i,i^2,\ldots,i^{\di})$, $i \in [\n]$. Here, $[\n]$ denotes the integer numbers from $1$ to $n$, i.e. the numbers $1$, $2$, \ldots, $n$. The size of this extended formulation is asymptotically significantly smaller than the size of the trivial ``vertex'' extended formulation, provided~$\di \ll (2\log_2 \n) / \big(1+\log_2 \log_2 (4n)\big)$.

As a possible application, consider the problem of minimizing a degree-$\di$ polynomial $p(t)$ over $t \in [n]$, where $\di$ is bounded. Our result implies that this can be formulated as a linear program with only a polylogarithmic number of constraints in $n$.

\subsection{Polytopes, Extensions and Extended Formulations}

Recall that a \emph{polytope}~$P\subseteq\R^{\di}$ is the convex hull of a finite point set~$V\subseteq\R^{\di}$. Without loss of generality, we assume that $P$ is full-dimensional, that is, $\dim P = \di$. Then $P$ can alternatively be described as the solution set of a system of finitely many linear inequalities, i.e.
\begin{displaymath}
    P=\setDef{x \in \R^{\di}}{\scalProd{a_j} {x}\le b_j,\, j\in [m]}
\end{displaymath}
for some~$a_j\in\R^{\di}$ and $b_j\in \R$, $j\in [m]$. The \emph{size} of the above linear description is the number of inequalities in the system, i.e.\ the number~$\m$. (When $P$ is not full-dimensional its linear description may contain linear equations. They are \emph{not} taken into account in the computation of the size.)

A polytope~$Q\subseteq \R^{\otherdi}$ together with an affine map~$\pi:\R^{\otherdi}\rightarrow \R^{\di}$ is called an \emph{extension} of polytope~$P \subseteq \R^{\di}$ if~$\pi(Q) = P$. The \emph{size} of extension~$Q$ is defined as the number of facets of $Q$, that is, the minimum number of inequality constraints in a linear description of $Q$.

By choosing appropriately the origin and the basis vectors of $\R^{\otherdi}$, we may assume that the projection $\pi$ is given by $\pi(x,y) := x$. In this case, the \emph{extended formulation} determined by the extension is simply a linear description of $Q$.
Extended formulations and extensions are two basically interchangeable concepts. 

The \emph{extension complexity} of polytope $P$ is the minimum size of an extension of $P$, or equivalently the minimum size of an extended formulation of $P$. This is denoted by $\xc(P)$. Thus $\xc(P)$ is the minimum number of facets of a polytope that projects to $P$.

It is a nontrivial problem to determine the extension complexity of a given polytope~$P$. In the first place, $P$ has an infinite number of extensions. Fortunately, it suffices to look at one single matrix for computing $\xc(P)$. The \emph{slack matrix} of polytope $P$ relative to point set~$V=\{v_1,\ldots, v_n\} \subseteq \R^{\di}$ with $P = \conv(V)$ and linear description $\scalProd{a_j} {x}\le b_j,\, j\in [m]$ of~$P$ is the matrix~$M\in\R^{\n \times \m}$ defined as\footnote{Throughout the paper we use superindices for dimensions, subindices for enumerating indices. We use parentheses to refer to a row or an element of a matrix, depending on the number of indices enclosed in the parentheses.}
\begin{displaymath}
    M(i,j):= b_j-\scalProd{a_j}{v_i}\,.
\end{displaymath}
That is, the entry corresponding to point~$v_i$ and inequality~$\scalProd{a_j}{x} \le b_j$ equals the slack~$b_j-\scalProd{a_j}{v_i}\ge 0$. We say that nonnegative vectors~$\alpha_i\in\R_{+}^r$ for $i \in [\n]$ and~$\beta_j\in\R_{+}^r$ for $j \in [\m]$ form a rank-$r$ \emph{nonnegative factorization} of the matrix~$M$ if the equation
\begin{displaymath}
    M(i,j)=\scalProd{\alpha_i}{\beta_j}
\end{displaymath}
holds for all~$i\in [\n]$ and~$j \in [\m]$. The \emph{nonnegative rank} of~$M$ is the minimum rank of a nonnegative factorization of~$M$, and is denoted $\nnegrk(M)$. 
The connection between nonnegative rank and extension complexity was provided by Yannakakis~\cite{Yannakakis}: $\xc(P) = \nnegrk(M)$, given that $\dim P \ge 1$. In other words, the minimum size of an extension of a polytope equals the nonnegative rank of any of its slack matrices.

In the present paper, we mainly work with nonnegative factorizations for slack matrices to guarantee the existence of extended formulations of certain size. 


\section{Cyclic Polytopes}

In this section, we define cyclic polytopes and list some of their properties. For more detailed information on cyclic polytopes we refer the reader to~\cite{Gruenbaum}, \cite{Ziegler}.

For $n\in\N$ and ~$\di \in\N$, $2\le \di$ let us define the corresponding {\it cyclic polytope}~$P^{\di}_{n}\subseteq\R^{\di}$ as the convex hull of the points~$v_i:=(i,i^2,\ldots,i^{\di})$, $i\in[n]$, i.e.
\begin{displaymath}
  P^{\di}_{n}:=\conv\setDef{(i,i^2,\ldots,i^{\di})}{i\in [n]}\,.
\end{displaymath}

We assume $n > \di$, then the cyclic polytope~$P^{\di}_{n}$ has $n$~vertices and dimension~$\di$. Moreover, the polytope~$P^{\di}_{n}$ is simplicial, i.e. every facet contains exactly $\di$~vertices. It is known that, a set of $\di$~vertices~$v_i$, $i \in S$ where~$S\subseteq [n]$ defines a facet if and only if it satisfies  {\it Gale's evenness condition}: the cardinality of the set~$[\ell_1,\ell_2]\cap S$ is even for all~$\ell_1, \ell_2 \in [n]$, $\ell_1 < \ell_2$ with $\ell_1, \ell_2 \notin S$.  Therefore, the slack matrix for the cyclic polytope~$P^{\di}_{n}$ can be obtained as
\begin{equation}
\label{eq:slack}
M^{\di}_{n}(i,S) := \prod_{j \in S} |j-i|\,,
\end{equation}
where $i\in[n]$ and $S\subseteq [n]$, $|S|=d$ satisfies Gale's evenness condition.

\section{Extended Formulation}

In this section, we construct extended formulations for cyclic polytopes: starting from the two dimensional case, going to even dimensions and finally considering odd dimensions.

\subsection{The case $d=2$}

For the sake of exposition, let us introduce the next notation. For $t_1$, $t_2\in\Z$, $t_1<t_2$ and ~$\di \in\N$, $2\le \di$ let us define the corresponding {\it cyclic polytope}~$P^{\di}_{[t_1,t_2]}$
\begin{displaymath}
  P^{\di}_{[t_1,t_2]}:=\conv\setDef{(i,i^2,\ldots,i^{\di})}{i\in [t_1,t_2]\cap\Z}\,.
\end{displaymath}
Note, that for every $n\in\N$ we have $P^{\di}_{n}=P^{\di}_{[1, n]}$.

 It is not hard to see that for two pairs of integers~$(t_1,t_2)$ and $(k_1,k_2)$ the polytopes $P^{\di}_{[t_1,t_2]}$ and $P^{\di}_{[k_1,k_2]}$ are affinely isomorphic if and only if the equation $t_2-t_1=k_2-k_1$ holds. Indeed, if $t_2-t_1$ is not equal to $k_2-k_1$ the polytopes $P^{\di}_{[t_1,t_2]}$ and $P^{\di}_{[k_1,k_2]}$ have different number of vertices, and thus can not be affinely isomorphic. On the other hand, if $t_2-t_1$ equals $k_2-k_1$ then they have the same slack matrix because translating indices preserves the difference $j - i$ in \eqref{eq:slack}. Concretely, the following affine map defines an isomorphism between $P^{\di}_{[t_1,t_2]}$ and $P^{\di}_{[k_1,k_2]}$: $(x_1,\ldots,x_{\di}) \mapsto (y_1,\ldots,y_{\di})$, where
\begin{equation}\label{eq:affine_isomorphism} 
  y_i:=(k_1-t_1)^i+\sum_{j=1}^i \binom{i}{j}(k_1-t_1)^{i-j} x_j\,.
\end{equation}

\begin{lem}\label{lem:two_dim}
 For the polytope~$P^2_{n}$ there is an extension of size at most~$2\lfloor \log_2(n-1)\rfloor+2$.
\end{lem}

\begin{proof}
Due to~\eqref{eq:affine_isomorphism}, we can affinely transform the polytope~$P^2_{n}$ into  the polytope~$$P^2_{[ -(n-1)/2,  (n-1)/2 ]}$$  
if~$n$ is odd and into the polytope 
$$P^2_{[-n/2+1,n/2]}$$
 if~$n$ is even.

In turn, for every $k\in\Z$ the polytope~$P^2_{[-k,k]}$ can be represented as the convex hull of two polytopes~$P^2_{[-k,0]}$ and~$P^2_{[0,k]}$, i.e.
\begin{displaymath}
 P^2_{[-k,k]}=\conv(P^2_{[-k,0]}\cup P^2_{[0,k]})\,.
\end{displaymath}
A similar representation exists for the polytope~$P^2_{[-k+1,k]}$, $k\in\Z$, namely
\begin{displaymath}
 P^2_{[-k+1,k]}=\conv(P^2_{[-k+1,0]}\cup P^2_{[1,k]})\,.
\end{displaymath}

There is a particularly nice relationship between $P^2_{[-k,0]}$ and $P^2_{[0,k]}$: the polytope~$P^2_{[-k,0]}$ is an image of the polytope $P^2_{[0,k]}$ under the reflection with respect to the hyperplane $x_1=0$ (a map defined by the sign change of the first coordinate). A similar statements holds for the polytopes $P^2_{[-k+1,0]}$ and $P^2_{[1,k]}$: the polytope~$P^2_{[-k+1,0]}$ is an image of the polytope $P^2_{[1,k]}$ under a sheared transformation~$(x_1,x_2)\mapsto (1 - x_1, x_2 - 2 x_1 + 1)$.

This fact allows us to use reflection relations (see Theorem~2 in~\cite{KaibelPashkovich}) and~\eqref{eq:affine_isomorphism} to state that every size-$f$ extended formulation of the polytope $P^2_{\lceil n/2 \rceil}$ leads to a size-$(f+2)$ extended formulation of the polytope $P^2_{n}$. In particular, the  extended formulation for $P^2_{[-k,k]}$ can be given as follows
$$
P^2_{[-k,k]}=\setDef{(x_1,x_2)\in\R^2}{\exists z_1 \text{ such that }(z_1,x_2)\in P^2_{[0,k]}\,\text{and}\, -z_1 \le x_1\le z_1}\,.
$$
An extended formulation for $P^2_{[-k+1,k]}$ can be constructed analogously (see remarks after Proposition~3 in~\cite{KaibelPashkovich})
\begin{align*}
P^2_{[-k+1,k]}=\{ (x_1,x_2)\in\R^2\,:\,\exists z_1,z_2 \text{ such that }(z_1,z_2)\in P^2_{[1,k]}\\
\text{and}\quad z_2-z_1=x_2-x_1, 2-3z_1+z_2 \le x_1+x_2\le z_1+z_2\}\,.
\end{align*}
Thus,  $\xc(P^2_{2k}) \le \xc(P^2_{k}) + 2$ and $\xc(P^2_{2k-1}) \le \xc(P^2_{k}) + 2$.

Let us note that $\xc(P^2_{n}) \le 2 \lfloor \log_2(n-1) \rfloor + 2$ can be easily verified for $n = 3,4,5,6$. Hence, to finish the proof it suffices to use the inequalities below
$$\xc(P^2_{2k}) \le \xc(P^2_{k}) + 2 \le  2 \lfloor \log_2(k-1) \rfloor + 4 =2 \lfloor \log_2 (2k-2) \rfloor + 2$$
and
$$\xc(P^2_{2k-1}) \le \xc(P^2_{k}) + 2 \le 2 \lfloor \log_2(k-1) \rfloor + 4 = 2 \lfloor \log_2 \big((2k-1) - 1\big) \rfloor + 2\,.$$
Therefore, we have
$$\xc(P^2_{n}) \le  2\lfloor \log_2(n-1)\rfloor+2$$
for all $n\ge 3$.
\end{proof}

We would like to note, that an explicit factorization of the slack matrix $M^2_n$ can be constructed in a similar way as the factorization in  \cite{FRT} .

\subsection{The case $\di=2q$}

Unfortunately, in the three dimensional case and higher the relationship between the polytopes~$P^{\di}_{[-k,0]}$ and~$P^{\di}_{[0,k]}$ is more complicated: the polytope~$P^{\di}_{[-k,0]}$ is an image of the polytope~$P^{\di}_{[0,k]}$ under the sign change of all coordinates with an odd index. This does not correspond to a symmetry with respect to a hyperplane, and thus we are not able to use reflection relations here. 

However, we may use the nonnegative factorization  for the slack matrix of~$P^2_{n}$ guaranteed by Lemma~\ref{lem:two_dim} together with Yannakakis' theorem~\cite{Yannakakis}, in order to prove that there is a nonnegative factorization of the slack matrix for~$P^{2q}_{n}$ of certain size. Let us consider even dimensions first, i.e. assume~$\di$ to be equal to~$2q$, $q\in \N$.

For two matrices~$A$ and $B$ of the same size, we define the elementwise or Hadamard product~$A \circ B$ by the following equation~$(A \circ B)(i,j):=A(i,j) B(i,j)$. We will need the following folklore result. 

\begin{lem}
\label{lem:Kronecker}
For all nonnegative matrices $A$, $B$ with the same number of rows and columns:
$$
\nnegrk(A \circ B) \le \nnegrk(A) \nnegrk(B)
$$
\end{lem}
\begin{proof}
If matrices~$A, B$ admit nonnegative factorizations given by vectors~$\alpha_i \in \R_+^{r}$, $\beta_j\in\R_+^{r}$ and vectors~$\gamma_i\in\R_+^{s}$, $\zeta_j\in\R_+^{s}$ respectively, then the matrix~$A\circ B$ has a size-$(rs)$ nonnegative factorization defined by the vectors~$\alpha_i\otimes\gamma_i\in\R_+^{r \times s}$ and $\beta_j\otimes\zeta_j\in\R_+^{r \times s}$. Here, for every two vectors $\mu\in\R^{r}$ and $\tau\in\R^{s}$ the vector~$\mu\otimes\tau$ lies in $\R^{r \times s}$ and is defined as $(\mu\otimes\tau)(i,j):=\mu(i)\tau(j)$.
\end{proof}

\begin{lem}\label{lem:even_case}
  The polytope~$P^{2q}_{n}$ has an extension of size at most~$\big(\xc(P^2_{n})\big)^q$.
\end{lem}
\begin{proof}
We construct $q$~matrices~$C_1$,\ldots,$C_q$ of suitable dimension such that the elementwise product $C_1 \circ \cdots \circ C_q$ equals the slack matrix~$M^{2q}_{n}$. In order to do that note that every set~$S\subseteq[n]$, $|S|=2q$ satisfying Gale's evenness condition can be partitioned into $q$~pairs $S_1$,\ldots,$S_q$, where each pair is equal to~$\{1,n\}$ or consists of two consecutive integers from~$[n]$. Every set~$S_r$, $1\le r\le q$ also satisfies Gale's evenness condition and consists of two elements, thus for every set~$S_r$ there is a corresponding column in $M^2_{n}$ and 
\begin{displaymath}
  M^2_{n}(i,S_r)=\prod_{j\in S_r} |j-i|\,.
\end{displaymath}
Now, define the entries of the matrices~$C_1$,\ldots,$C_q$ in the column indexed by the set~$S$ as~$C_r(i,S):=M^2_{n}(i,S_r)$. Notice that
\begin{displaymath}
(C_1 \circ \cdots \circ C_q)(i,S) = \prod_{r = 1}^q \prod_{j \in S_r} |j-i| = \prod_{j \in S} |j-i| = M^{2q}_{n}(i,S)\,.
\end{displaymath}

Finally, it is straightforward to verify that the matrices~$C_1$,\ldots,$C_q$ are obtained from~$M^2_{n}$ by duplicating, deleting and reordering columns, and thus the nonnegative rank of every of these matrices is bounded from above by the nonnegative rank of the matrix~$M^2_{n}$. Hence, by Lemma~\ref{lem:Kronecker} the slack matrix~$M^{2q}_{n}$ admits a nonnegative factorization of size at most~$\big(\xc(P^2_{n})\big)^q$.
\end{proof}

\subsection{The case $\di=2q+1$}

\begin{lem}\label{lem:odd_case} 
The polytope~$P^{2q+1}_{n}$ has an extension of size at most~$2 \xc(P^{2q}_{n-1})$.
\end{lem}
\begin{proof}

Let us prove that $\xc(P^{2q+1}_{n}) \le \xc(P^{2q}_{[2,n]}) + 
\xc(P^{2q}_{[1,n-1]}) = 2 \xc(P^{2q}_{n-1})$.

 Every set~$S\subseteq[n]$, $|S|=2q+1$ satisfying Gale's evenness condition has one of the following forms:
\begin{enumerate}
 \item \label{case:first_element} $1\in S$ and the set~$S\setminus\{1\}$ defines a facet of the polytope~$P^{2q}_{[2,n]}$
 \item \label{case:last_element} $n\in S$ and the set~$S\setminus\{n\}$ defines a facet of the polytope~$P^{2q}_{[1,n-1]}$.
\end{enumerate}
The columns indexed by the sets~$S$ satisfying the condition~\eqref{case:first_element} form a matrix, which is equal to the matrix~$M^{2q}_{[2,n]}$, where the row~$M^{2q}_{[2,n]}(i)$, $i\in [2,n]\cap Z$ is scaled by the positive scalar~$i-1$, plus an appended zero row corresponding to the index~$i=1$. Thus, the nonnegative rank of the submatrix of~$M^{2q+1}_{n}$ indexed by sets~$S$ satisfying the condition~\eqref{case:first_element} and the nonnegative rank of~$M^{2q}_{[2,n]}$ are equal. We can estimate the nonnegative rank of the submatrix of~$M^{2q+1}_{n}$ indexed by sets~$S$ satisfying the condition~\eqref{case:last_element} in a similar way.
\end{proof}

Finally, Lemmas~\ref{lem:two_dim},~\ref{lem:even_case} and~\ref{lem:odd_case} together lead to our theorem.
\begin{thm}
\label{thm:main}
  The polytope~$P^{\di}_{n}$ has an extension of size at most~$2(2\lfloor \log_2(n-1)\rfloor+2)^{\lfloor \di/2 \rfloor}$.
\end{thm}

\section{Concluding Remarks}

We remark that Lemma~\ref{lem:two_dim} crucially uses the fact that $P^{2}_{n}$ is the convex hull of the points $(i,i^2)$ for $n$ \emph{consecutive} integers $i \in [n]$. Actually, \cite{FRT} prove a $\Omega(\sqrt{n} / \sqrt{\log n})$ lower bound on the worst-case extension complexity of a $2$-dimen\-sional cyclic polytope of the form $P = \conv \{ (i,i^2) : i \in X\}$ where $X \subseteq [2n]$ and $|X| = n$. 

Finally, there seems to be currently no lower bound on the extension complexity of $P^d_{n}$ that would match the upper bound given by Theorem~\ref{thm:main}. For instance, it follows from~\cite{FKPT} that the best lower bound that only relies on the combinatorial structure is $O(d^2 \log n)$.

\bigskip

We would like to thank the anonymous referees for their comments which led to a better exposition of the results in the present paper.

\end{document}